\documentclass[10pt, a4paper]{amsart}
\usepackage{amssymb, amsmath, amsfonts, mathrsfs, amsthm}
\usepackage[utf8]{inputenc}
\usepackage{lmodern}
\usepackage[T1]{fontenc}
\usepackage[a4paper,includeheadfoot,margin=2.54cm]{geometry}
\usepackage[matrix, arrow, curve]{xy}

\DeclareMathOperator{\rk}{\mathrm{rk}}

\DeclareMathOperator{\red}{\mathrm{red}}
\DeclareMathOperator{\alg}{\mathrm{alg}}

\DeclareMathOperator{\OO}{\mathscr{O}}
\DeclareMathOperator{\PP}{\mathbb{P}}

\DeclareMathOperator{\Sing}{\mathrm{Sing}}
\DeclareMathOperator{\QQ}{\mathbb{Q}}

\DeclareMathOperator{\FF}{\mathscr{F}}
\DeclareMathOperator{\GG}{\mathscr{G}}
\DeclareMathOperator{\Var}{\mathrm{Var}}
\DeclareMathOperator{\reg}{\mathrm{reg}}
\DeclareMathOperator{\Zar}{\mathrm{Zar}}
\DeclareMathOperator{\inv}{\mathrm{inv}}

\begin{document} 

\title{Towards classification of codimension 1 foliations on threefolds of general type}
\date{}
\author{Aleksei Golota}
\thanks{This work is supported by the Russian Science Foundation under grant 18-11-00121.}

\newtheorem{theorem}{Theorem}[section] 
\newtheorem*{ttheorem}{Theorem}
\newtheorem{lemma}[theorem]{Lemma}
\newtheorem{proposition}[theorem]{Proposition}
\newtheorem*{conjecture}{Conjecture}
\newtheorem{corollary}[theorem]{Corollary}

{\theoremstyle{remark}
\newtheorem{remark}[theorem]{Remark}
\newtheorem{example}[theorem]{Example}
\newtheorem{notation}[theorem]{Notation}
\newtheorem{question}{Question}
}

\theoremstyle{definition}
\newtheorem{construction}[theorem]{Construction}
\newtheorem{definition}[theorem]{Definition}

\begin{abstract} We aim to classify codimension 1 foliations $\FF$ with canonical singularities and $\nu(K_{\FF}) < 3$ on threefolds of general type. We prove a classification result for foliations satisfying these conditions and having non-trivial algebraic part. We also describe purely transcendental foliations $\FF$ with the canonical class $K_{\FF}$ being not big on manifolds of general type in any dimension, assuming that $\FF$ is non-singular in codimension $2$.
\end{abstract}

\maketitle

\section{Introduction}

The Minimal Model Program (MMP) is one of the main guiding principles in the study of algebraic varieties. It predicts that every projective variety $X$ with mild singularities should have a birational model $X'$ which is also mildly singular and either has $K_{X'}$ nef (minimal model) or admits a Mori fiber space. So, informally speaking, all algebraic varieties should be constructed, using birational equivalences and fibrations, from those of the three classes: $K_X$ positive (general type), $K_X$ numerically trivial (Calabi--Yau) and $K_X$ negative (Fano).

The works of Brunella \cite{Bru97, Bru03, Bru15}, Mendes \cite{Men00} and McQuillan \cite{McQ08} established a version of the MMP for holomorphic foliations on surfaces. Analogously to the case of varieties, for a foliation $\FF$ with mild singularities it is possible to define the canonical class $K_{\FF}$ and its invariants such as Kodaira and numerical dimension. Foliations on surfaces admit a classification according to these invariants; moreover, this classification is rather explicit for foliations ``not of general type''. The final step in the classification of minimal models of foliations on surfaces is the following theorem (see \cite[Theorem on p. 122]{Bru03} and \cite[Section IV.5]{McQ08}).

\begin{theorem}\label{BruMcQ} Let $\FF$ be a foliation with reduced singularities on a surface $S$ of general type. Suppose that $K_{\FF}$ is nef and not big. Then either 
\begin{enumerate}
\item $\FF$ is algebraically integrable and induced by an isotrivial fibration, $\kappa(K_{\FF}) = \nu(K_{\FF}) = 1$;
\item $\FF$ is transcendental and the pair $(S, \FF)$ is isomorphic to a Hilbert modular surface with a Hilbert modular foliation, $\kappa(K_{\FF}) = -\infty, \nu(K_{\FF}) = 1$.
\end{enumerate}
\end{theorem}

The setting of the above theorem is interesting from several points of view. It involves the interplay between positivity of $K_X$ and $K_{\FF}$ (see e.g. \cite{CP15, CP19} for related topics); the failure of abundance for foliations (see \cite{McQ08, Tou16}). Another context, in which these foliations appear, is the study of Kobayashi hyperbolicity (see e.g. \cite{Kob70, Dem97}). Entire curves tangent to foliations on surfaces have been studied in connection with the Green--Griffiths conjecture (see \cite{McQ98, Bru99, DR15}). In higher dimensions we have, for example, the following theorem \cite[Theorem F]{GPR13}.

\begin{theorem} Let $\FF$ be a codimension one foliation with canonical singularities on a projective threefold $X$. Suppose that there exists a generically nondegenerate meromorphic map $f \colon \mathbb{C}^2 \dasharrow X$ such that the image $f(\mathbb{C}^2)$ is Zariski-dense in $X$ and tangent to $\FF$. Then the canonical class of $\FF$ is not big.
\end{theorem}

Recently, in a series of works \cite{Spi20, CS21, SS19} the MMP for codimension $1$ foliations on threefolds has been established. Moreover, some structure theorems for codimension $1$ foliations are known in higher dimensions, for example in the case $K_{\FF} \equiv 0$ \cite{LPT18, Dru21} or $-K_{\FF}$ ample \cite{AD13, AD17, AD19}. These results suggest that the MMP-type classification for codimension $1$ foliations should hold in any dimension. 

Motivated by the above results, we would like to describe codimension $1$ foliations $\FF$ with canonical singularities and $\nu(K_{\FF}) < 3$ on threefolds of general type. In this paper we are able to prove the following partial analogue of Theorem \ref{BruMcQ}.

\begin{ttheorem}[A] Let $\FF$ be a codimension one foliation with canonical singularities on a smooth projective threefold $X$ of general type. Suppose that $K_{\FF}$ is not big and the algebraic rank $r_a(\FF)$ is positive. Then one of the following two cases occurs.
\begin{enumerate}
\item $\FF$ is algebraically integrable. Then there exists a generically finite morphism $f \colon X' \to X$ such that the threefold $X'$ is birational to the product of a surface $S$ and a curve $C$, and the foliation $f^{-1}\FF$ is birationally equivalent to the relative tangent bundle of the projection $S \times C \to C$.
\item The algebraic rank of $\FF$ is equal to $1$. Then there exists a birational morphism $f \colon X' \to X$ such that the threefold $X'$ admits a birationally isotrivial fibration $\pi \colon X' \to S$ to a surface $S$ and the foliation $f^{-1}\FF$ is the pullback via $\pi$ of a foliation $\GG$ on $S$ such that $\nu(K_{\GG}) = 1$. Moreover, if $\FF$ is non-abundant, then $\GG$ is birationally equivalent to a Hilbert modular foliation.
\end{enumerate}
\end{ttheorem}

At the moment we do not know how to classify purely transcendental foliations with canonical singularities and not of general type on threefolds of general type. However, if we restrict ourselves to foliations regular in codimension $2$, then we are able to prove a result which confirms our expectations. The starting point for us is another remarkable theorem by Brunella (see \cite[pp. 587-588]{Bru97}).

\begin{theorem} \label{brureg}
Let $\FF$ be a regular foliation on a minimal surface $X$ of general type. Then the conormal bundle $N^*_{\FF}$ is numerically effective.
\end{theorem}

This theorem essentially follows from Baum--Bott formula, Riemann--Roch theorem and intersection theory. Assuming that $\FF$ is non-singular in codimension $2$, we use the Baum--Bott formula and intersection computations in a similar way to study purely transcendental foliations in higher dimensions. Together with Lefschetz-type theorems, it allows us to ``lift'' positivity of the conormal bundle from hyperplane sections. Then we use classification results of Touzet \cite{Tou13, Tou16} to describe foliations from this class.

\begin{ttheorem}[B] Let $X$ be a smooth projective manifold of general type, $\dim(X) = n \geqslant 2$. Let $\FF$ be a codimension 1 foliation on $X$. Suppose that \begin{enumerate} 
\item $K_{\FF}$ is not big;
\item $\FF$ is purely transcendental;
\item $\mathrm{codim}_X\Sing(\FF) \geqslant 3$.
\end{enumerate}

Then the foliation $\FF$ is induced by a Hilbert modular foliation via a morphism $X \to M_H \simeq \mathbb{D}^N/\Gamma$, generically finite onto its image.
\end{ttheorem}

Our paper is organized as follows. In section \ref{prelim} we gather the information on Kodaira and numerical dimensions for $\QQ$-divisors, basic notions from foliation theory and foliated birational geometry. In section \ref{main1} we recall some important definitions and results concerning fibrations with fibers of general type. Then we use these results to prove Theorem (A), see Proposition \ref{AlgInt} and Theorem \ref{MainThm} below. In section \ref{main2} we prove Theorem (B) (see Theorem \ref{terminal}) and pose some questions for future research.

\subsection*{Acknowledgements} This paper is a part of a project on hyperbolicity and foliations we have been working on for a long time. We would like to thank our advisor Constantin Shramov for encouraging us to write down these results. We also thank Jorge Vit\'orio Pereira, Erwan Rousseau, Jean-Pierre Demailly, Misha Verbitsky and Vladimir Lazi\'c for very helpful discussions on various topics related to this work. We also thank the referee for very careful reading of this paper and numerous valuable comments.

\section{Preliminaries} \label{prelim}

\subsection{Kodaira and numerical dimensions} In this subsection we recall the notions of Kodaira and numerical dimension for $\QQ$-divisors. 

\begin{definition} Let $D$ be a $\QQ$-divisor on a normal projective variety $X$. The Kodaira dimension (or Kodaira--Iitaka dimension) of $D$ is defined as $$\kappa(X, D) = \max\{k \mid \limsup_{m \to \infty}\frac{h^0(X, \OO_X(\lfloor mD\rfloor)}{m^k} > 0\}$$ if $h^0(X, \OO_X(\lfloor mD\rfloor)) > 0$ for some $m$ and $\kappa(X, D) = -\infty$ otherwise.
\end{definition}

\begin{definition} \label{num} Let $X$ be a normal projective variety, $D$ an $\QQ$-divisor on $X$ and $A$ an ample divisor. We define the quantity $$\nu(D, A) = \max\left\{k \in \mathbb{Z}_{\geqslant 0} \mid \limsup_{m \to \infty}\frac{h^0(X, \OO_X(\lfloor mD\rfloor + A))}{m^k} > 0\right\}.$$ The numerical dimension of $D$ is defined to be $$\nu(D) = \max\{\nu(D, A) \mid \mbox{$A$ ample}\}.$$ If $D$ is not pseudoeffective then we set $\nu(D) = -\infty$.
\end{definition}

The main property of $\nu(D)$ is that it depends only on the numerical class of a divisor. 

\begin{proposition} \label{Numprop} We have the following basic properties of Kodaira and numerical dimensions (see \cite[Lemma II.3.11, Proposition III.5.7]{Nak04} and \cite[Sec. V]{Nak04} for proofs).
\begin{enumerate}
\item We have $\kappa(X, D) \leqslant \nu(X, D) \leqslant n$ for any $\QQ$-divisor $D$;
\item For a nef $\QQ$-divisor $D$ we have $\nu(X, D) = \max\{k \mid D^k \cdot A^{n-k} \neq 0\}$
\item If $D$ and $E$ are pseudoeffective $\QQ$-divisors then $\nu(D + E) \geqslant \max\{\nu(D), \nu(E)\};$ 
\item If $f \colon Y \to X$ is a surjective morphism then $\kappa(Y, f^*D) = \kappa(X, D)$ and $\nu(Y, f^*D) = \nu(X, D)$.
\item If $\varphi \colon \widetilde X \to X$ is a birational morphism and $D$ is a $\QQ$-divisor on $X$ then $$\kappa(X, D) = \kappa(\widetilde X, \varphi^*D + E) \quad \mbox{and} \quad \nu(X, D) = \nu(\widetilde X, \varphi^*D + E)$$ for any effective $\varphi$-exceptional $\QQ$-divisor $E$.
\end{enumerate}
\end{proposition}

We also recall the notion of \emph{movable intersection product} for pseudoeffective classes from \cite{BDPP13} (see also \cite[Sec. 4]{Leh13}).

\begin{theorem} \label{posprod} Let $X$ be a smooth projective variety of dimension $n$. Denote by $\mathcal{E}$ the cone of pseudoeffective $(1,1)$-classes on $X$. Then for every $k \in \{1, \ldots, n\}$ there exists a map $$\prod_{i=1}^k\mathcal{E}_i \to H^{k,k}_{\geqslant 0}(X, \mathbb{R}), \quad (L_1, L_2, \ldots, L_k) \to \langle L_1 \cdot L_2 \cdots L_k\rangle$$ called the movable intersection product, such that the following properties hold.
\begin{enumerate}
\item We have $\mathrm{vol}(L) = \langle L\rangle^n$;
\item The movable intersection product is increasing. homogeneous of degree $1$ and superadditive in each variable: $$\langle L_1 \cdots (M + N) \cdots L_k\rangle \geqslant \langle L_1 \cdots M\cdots L_k\rangle + \langle L_1 \cdots N\cdots L_k\rangle.$$
\item For $k=1$ the movable intersection product gives a divisorial Zariski decomposition: $$L = P(L) + N(L),$$ where $P(L) = \langle L\rangle$ is nef in codimension $1$.
\end{enumerate}
\end{theorem}

The following proposition, proved in \cite[Proposition 4.13]{Leh13}, shows that the movable intersection product depends only on positive parts in the divisorial Zariski decompositions.

\begin{proposition} \label{lehmann} Let $X$ be a smooth projective variety and $L_1, \ldots, L_k$ pseudoeffective classes on $X$. Then we have an equality $$\langle L_1 \cdot \cdots \cdot L_k\rangle = \langle P(L_1) \cdot \cdots \cdot P(L_k)\rangle,$$ where $P(L_i)$ are positive parts in the divisorial Zariski decompositions of $L_i$.
\end{proposition}

\begin{remark} This product was used in \cite[Definition 3.6]{BDPP13} to define another version of numerical dimension for a pseudoeffective class $D$:
$$\nu_{BDPP}(D) = \max\{k \mid \langle D\rangle^k \neq 0 \}.$$ The equivalence of this definition to Definition \ref{num} was claimed in \cite{Leh13}, see the subsequent corrections in \cite{E16, Les19}. Still, by the results in \cite[Sec. 6]{Leh13} we have an inequality $$\nu_{BDPP}(D) \leqslant \nu(D)$$ for any pseudoeffective $\QQ$-divisor $D$.
\end{remark}

\subsection{Foliations} In this subsection we recollect basic definitions and facts from foliation theory. General references for this subsection are e.g. \cite{AD13, Bru15}. We denote by $X, Y$ normal and $\mathbb{Q}$-factorial projective varieties over $\mathbb{C}$ and by $\FF$ (resp. $\GG$) coherent sheaves of $\OO_X$-modules (resp. $\OO_Y$-modules). If $\mathscr{E}$ is a coherent torsion-free sheaf of rank $r$ then by $\det(\mathscr{E})$ we denote the reflexive hull $(\bigwedge^{r}\mathscr{E})^{**}$. A subset $U \subset X$ is called \emph{big} if the codimension of $U$ in $X$ is at least $2$.

\begin{definition} A {\it foliation} on a variety $X$ is a coherent subsheaf $\FF \subset T_X$ which is \begin{itemize} \item saturated, that is, the quotient $T_X / \FF$ is torsion-free; \item closed under the Lie bracket, i. e. the map $ \bigwedge^2 \FF \to T_X /\FF$ is zero.\end{itemize} The rank of the foliation, denoted by $\rk(\FF)$, is the rank of the sheaf $\FF$ at a general point of $X$ and the codimension of $\FF$ is $q = \dim(X) - \rk(\FF)$.
\end{definition}

A foliation on a smooth variety $X$ is {\it regular} if $\FF$ is a subbundle of $T_X$ at every point of $X$. In general, let $X^{\reg}$ be the maximal open subset of $X$ such that $\FF|_{X^{\reg}}$ is regular. Then $X^{\reg}$ is a big subset of $X$; the complement $X \setminus X^{\reg}$ is denoted by $\Sing(\FF)$ and is a closed subscheme of codimension at least $2$ (this follows from the fact that $\FF$ is saturated). 

\begin{remark} The {\it normal sheaf} to $\FF$ is defined to be $N_{\FF} = (T_X/\FF)^{**}$. Taking the $q$-th wedge power of the map $N^{*}_{\FF} \hookrightarrow \Omega_X^1$ gives rise to a $q$-form $\omega_{\FF} \in H^0(X, \Omega^1_X \otimes \det N_{\FF})$ with zero locus of codimension at least two. This twisted $q$-form is locally decomposable and integrable, which means that locally in the Euclidean topology around a general point of $X$ we have $$\omega_{\FF} = \omega_1 \wedge \omega_2 \wedge \cdots \wedge \omega_q$$ for locally defined $1$-forms $\omega_1, \ldots, \omega_q$ satisfying the integrability condition $d\omega_i \wedge \omega_{\FF} = 0$. Conversely, any locally decomposable and integrable twisted $q$-form $\omega \in H^0(X, \Omega^q_X \otimes \mathscr{L})$ defines a codimension $q$ foliation. The subsheaf $\FF \subsetneq T_X$ is obtained as the kernel of the morphism $c_{\omega} \colon T_X \to \Omega^{q-1}_X \otimes \mathscr{L}$ given by contraction with $\omega$.
\end{remark}

\begin{definition}[Tangent and transverse subvarieties] Let $W$ be an irreducible and reduced subvariety of $X$ not contained in $\Sing(\FF)$. We say that $W$ is {\it tangent} to $\FF$ if tangent space $T_W$ factors through $\FF$ at all points $x \in X^{\reg}$. A subvariety is {\it transverse} to $\FF$ if it is not tangent to $\FF$. If $\FF$ factors through the tangent space of $W$ then $W$ is called {\it invariant} by $\FF$. In the special case of a codimension $1$ foliation we call a hypersurface tangent to $\FF$ an {\it invariant hypersurface}.
\end{definition}

\begin{definition}\label{pullback} We now define the following constructions for a foliation $\FF$: \begin{itemize} \item{(Pullback)} Let $f \colon Y \dasharrow X$ be a dominant rational map of varieties restricting to a morphism $f^{\circ} \colon Y^{\circ} \to X^{\circ}$ where $Y^{\circ} \subset Y$ and $X^{\circ} \subset X$ are Zariski open subsets. Let $\FF$ be a foliation on $X$ given by a $q$-form $\omega_{\FF} \in H^0(X, \Omega^q_X \otimes \det N_{\FF})$. Then we have an induced $q$-form $$\omega_{Y^{\circ}} \in H^0(Y^{\circ}, \Omega_{Y^{\circ}}^q \otimes (f^{\circ})^*(\det N_{\FF}|_{X^{\circ}}))$$ which defines a foliation $\GG$ on $Y^{\circ}$. We define the {\it pulled-back} foliation $f^{-1}\FF$ to be the saturation of $\GG$ in $T_Y$.

\item{(Projection)} Let $f \colon Y \to X$ be a dominant morphism and let $\GG$ be a foliation on $Y$. We have an induced map $df \colon T_Y \to f^*T_X$. The foliation $\GG$ is called {\it projectable} under $f$ if for a general point $x \in X$ the image $df_{y}(\GG_y)$ does not depend on the choice of $y \in f^{-1}(x)$ and $\dim(df_{y}(\GG_y)) = r = \rk(\GG)$. In particular, if $f \colon Y \to X$ is a birational contraction then any foliation $\GG \subsetneq T_Y$ is projectable under $f$ and induces a foliation $\FF = f_*\GG$ of the same rank on $X$.

\item{(Restriction)} Let $Z \subseteq X$ be a smooth subvariety transverse to a foliation $\FF$ given by a twisted $q$-form $\omega_{\FF} \in H^0(X, \Omega^q_X \otimes \det N_{\FF})$. Suppose that the restriction of $\omega_{\FF}$ to $Z$ is nonzero. Then we obtain a nonzero induced $q$-form $\omega_Z \in H^0(Z, \Omega^q_Z \otimes \det N_{\FF}|_Z)$. Let $B$ be the maximal effective divisor on $Z$ such that $$\omega_Z \in H^0(Z, \Omega^q_Z \otimes \det N_{\FF}|_Z(-B)).$$ This $q$-form defines a codimension $q$ foliation on $Z$ which is called {\em the restriction of $\FF$ to $Z$} and denoted $\FF_Z$. 
\end{itemize}
\end{definition}

\begin{definition} Let $\FF$ be a foliation on a smooth variety $X$. Since the integrability condition holds for $\FF$, the Frobenius theorem implies that for every point $x \in X^{\reg}$ there exists an open neighbourhood $U = U_x$ and a submersion $p \colon U \to V$ such that $\FF|_U = T_{U/V}$. A {\it leaf} of $\FF$ is a maximal connected, locally closed submanifold $L \subset X^{\reg}$ such that $\FF|_L = T_L$. A leaf $L$ is {\it algebraic} if it is open in its Zariski closure or, equivalently, if $\dim(L) = \dim({\overline L}^{\Zar})$. In this case we use the word ``leaf'' for the Zariski closure of a leaf as well.
\end{definition}

\begin{definition}[Algebraic integrability] A foliation $\FF$ on $X$ is called {\it algebraically integrable} or simply {\it algebraic} if the leaf of $\FF$ through a general point of $X$ is algebraic. An example of an algebraically integrable foliation is given by the relative tangent sheaf $\FF = (T_{X/Y})$ of a fibration $f \colon X \to Y$ where $\dim(Y) < \dim(X)$. The leaves of $\FF$ in this case are just the fibers of $f$. We say that $\FF$ is {\it induced by the fibration} $f$. 
\end{definition}

\begin{proposition}[Rational first integrals] \label{ratint} Let $\FF$ be an algebraically integrable foliation on $X$. Then there is a unique irreducible subvariety $W$ of the cycle space $\mathrm{Chow}(X)$ parameterizing the closure of a general leaf of $\FF$. Let $V \subset W \times X$ be the universal cycle with universal morphisms $\pi \colon V \to W$ and $e \colon V \to X$. Then the morphism $e$ is birational and for $w \in W$ general $e(\pi^{-1}(w)) \subset X$ is the closure of a leaf of $\FF$. The normalization $\widetilde{W}$ of $W$ is called the {\it space of leaves} of $\FF$ and the induced rational map $X \dasharrow \widetilde{W}$ is called {\it rational first integral} of $\FF$. In other words, an algebraically integrable foliation is induced by a fibration on a suitable birational model of $X$.
\end{proposition}

\begin{proposition}[Algebraic and purely transcendental parts] \label{algred} Let $\FF$ be a foliation on $X$. Then there exists a normal variety $Y$ together with a dominant rational map $\varphi \colon X \dasharrow Y$ with connected fibers and a foliation $\GG$ on $Y$ such that \begin{itemize} \item The foliation $\GG$ is {\it purely transcendental}, that is, there is no positive-dimensional subvariety tangent to $\GG$ through a general point of $Y$; \item We have $\FF = \varphi^{-1}\GG$.\end{itemize} The pair $(Y, \GG)$ is unique up to birational equivalence. The foliation induced by $\varphi$ is called the {\it algebraic part} of $\FF$ and denoted by $\FF^{\alg}$. The rank of $\FF^{\alg}$ is called the {\it algebraic rank} of $\FF$; by construction, it is a birational invariant.
\end{proposition}

\begin{proposition}[Baum--Bott formula, \cite{BP06}] \label{BaumBott} Let $\FF$ be a codimension $1$ foliation on a complex manifold $X$ of dimension at least $2$. We have the following equality in $H^4(X, \mathbb{C})$: $$c^2_1(N_{\FF}) = \sum_{Y}BB(\FF, Y)[Y],$$ where $Y$ ranges over irreducible components of $\mathrm{Sing}(\FF)$ of codimension $2$, and $BB(\FF, [Y])$ is a number, called the Baum--Bott index of $\FF$ at $Y$.
\end{proposition}

\subsection{Canonical class and singularities of foliations}

\begin{definition}By the {\it canonical class} of a foliation $\FF$ we denote a linear equivalence class of Weil divisors $K_{\FF}$ such that $\OO_X(K_{\FF}) \cong \det(\FF)^*$. We have the relation $$ K_X = K_{\FF} + \det(N^*_{\FF}).$$
\end{definition}

\begin{definition} Let $f \colon Y \to X$ be a birational morphism and let $\FF$ be a foliation on $X$. Then we have an induced foliation $\widetilde\FF = f^{-1}\FF$ on $Y$. We express the canonical divisor $K_{\widetilde\FF}$ as $$ K_{\widetilde\FF} = f^*K_{\FF} + \sum_ia(E_i, \FF, X)E_i$$ where the sum is over all prime $f$-exceptional divisors on $Y$. The foliation $\FF$ is said to have {\it canonical (resp., terminal) singularities} if all discrepancies $a(E_i, \FF, X)$ are nonnegative (resp., positive) for any such birational morphism. 
\end{definition}

\begin{remark} By property (5) in Proposition \ref{Numprop}, if $\FF$ is a foliation with canonical singularities and $\varphi \colon (\widetilde{X}, \widetilde{\FF}) \to (X, \FF)$ is a birational morphism, then $$\kappa(K_{\FF}) = \kappa(K_{\widetilde\FF}) \quad \mbox{and} \quad \nu(K_{\FF}) = \nu(K_{\widetilde\FF}).$$ Thus, the class of foliations with canonical singularities is natural to consider in birational geometry. A famous theorem of Seidenberg \cite{Sei68} says that any foliation singularity on a smooth surface can be transformed to a \emph{reduced} singularity by a finite number of blow-ups. For codimension one foliations in dimension $3$ there is a resolution theorem of Cano \cite{Can04}, in terms of so-called \emph{simple} foliation singularities. Both reduced and simple foliation singularities are canonical (see \cite[Fact I.2.4]{McQ08} and \cite[Lemma 2.9]{CS21}, respectively).
\end{remark}

\begin{remark} Unlike canonical ones, terminal foliation singularities form a rather restricted class. A result of McQuillan \cite[Corollary I.2.2]{McQ08} says that a germ of a normal $\QQ$-Gorenstein surface with a terminal foliation is isomorphic to a finite cyclic quotient of a smooth surface germ with a regular foliation. For more details on terminal foliation singularities on threefolds see \cite[Section 5]{SS19}.
\end{remark}

To express the canonical class of an algebraic foliation we recall the notion of ramification (see e.g. \cite[Definition 2.5]{AD19}).

\begin{definition} \label{ramif} Let $f \colon X \dasharrow Y$ be a dominant rational map between normal and $\QQ$-factorial projective varieties. Let $Y^{\circ} \subset Y$ be a maximal open subset such that $f^{\circ} = f|_{f^{-1}(Y^{\circ})} \colon f^{-1}(Y^{\circ}) \to Y^{\circ}$ is an equidimensional morphism. Define $$R(f^{\circ}) = \sum_{D}\left ((f^{\circ})^*D - ((f^{\circ})^*D)_{\mathrm{red}}\right),$$ where the sum is over all prime divisors $D$ on $Y^{\circ}$. The ramification divisor $R(f)$ of $f$ is defined as the Zariski closure of $R(f)$ in $X$. 
\end{definition}

We also recall the construction of flattening (see \cite[Theorem 1]{Ray72}).

\begin{proposition}\label{flatten} Let $\pi \colon X \to Y$ be a morphism of projective varieties. Then there exists a birational morphism $\varphi \colon Y' \to Y$ from a smooth projective variety $Y'$ and a resolution $X'$ of the main component of the fiber product $X \times_{Y} Y'$ such that in the diagram 
\begin{equation}\label{flatdia}
\xymatrix{
X' \ar[r]^{\pi'} \ar[d]_{\psi} & Y' \ar[d]^{\varphi}\\
X \ar[r]^{\pi} & Y
}
\end{equation}

the induced morphism $\pi' \colon X' \to Y'$ is flat and moreover every $\pi'$-exceptional divisor is $\psi$-exceptional.
\end{proposition}

The following two propositions (see \cite[2.6]{AD19}) provide formulas for the canonical class of the inverse image of a foliation by an equidimensional morphism.

\begin{proposition} \label{canclassalg} Let $\FF$ be a foliation induced by an equidimensional morphism $\pi \colon X \to Y$. Then the canonical class of $\FF$ is given by the formula $$K_{\FF} = K_{X/Y} - R(f).$$ More generally, if $\FF$ is induced by a (not necessarily equidimensional) morphism $\pi \colon X \to Y$ then there exist birational morphisms $\psi \colon X' \to X$ and $\varphi \colon Y' \to Y$ and a morphism $\pi' \colon X' \to Y'$ with connected fibers as in the diagram \eqref{flatdia} such that every $\psi$-exceptional divisor is $\pi'$-exceptional and the canonical class of $\FF$ is given by the formula $$K_{\FF} = \psi_*(K_{X'/Y'} - R(\pi')).$$
\end{proposition}

\begin{proposition} \label{canclasspull} Let $\pi \colon X \to Y$ be an equidimensional morphism and let $\GG$ be a foliation on $Y$. Denote by $\FF$ the pullback of $\GG$ via $\pi$. Then the canonical class of $\FF$ can be expressed as $$K_{\FF} = \pi^*K_{\GG} + K_{X/Y} - \sum_i(\pi^*B_i - (\pi^*B_i)_{\red})$$ where $B_i$ are $\GG$-invariant prime divisors in the critical locus of $\pi$.
 \end{proposition}

We also state the Hurwitz formula for codimension one foliations (see e.g. \cite[Proposition 3.7]{Spi20}).

\begin{proposition} \label{hurwitz} Let $f \colon \overline X \to X$ be a finite surjective morphism of projective varieties. Let $\FF$ be a codimension $1$ foliation on $X$ and denote by $\overline\FF$ the induced foliation on $\overline X$. Then the canonical classes of $\FF$ and $\overline\FF$ are related by the formula $$K_{\overline\FF} = f^*K_{\FF} + \sum_D \epsilon(D)(r_D - 1)D.$$ Here the sum is over prime divisors $D$ on $\overline X$ with ramification index $r_D$, and $\epsilon(D)$ is zero if $D$ is $\overline\FF$-invariant and $1$ otherwise.

In particular, if $f \colon \overline{X} \to X$ is a ramified cover with $\FF$-invariant branch divisor then $K_{\overline \FF} = f^*K_{\FF}$.
\end{proposition}

We will need an adjunction formula for foliations induced on general hyperplane sections on a projective variety (see \cite[Proposition 3.6]{Dru21} and \cite[Lemma 2.9]{AD19} for the proof).

\begin{proposition} \label{adj} Let $X$ be a smooth projective variety of dimension at least 3 and let $\FF$ be a codimension $1$ foliation on $X$. Let $D \in |L|$ be a general divisor in a very ample linear system on $X$. Then $\FF$ induces a codimension $1$ foliation $\FF_D$ on $D$ such that $$K_{\FF_D} = (K_{\FF} + D)|_D \quad \mbox{and} \quad N^*_{\FF_D} = (N^*_{\FF})|_D.$$
\end{proposition}

\section{Foliations with positive algebraic rank} \label{main1}

\subsection{Families with general fibers of general type} We start by recalling a few facts about fibrations on varieties of general type.

\begin{definition} A fibration is a proper and surjective morphism $\pi \colon X \to Y$ between normal projective varieties, such that the fibers of $\pi$ are connected or, equivalently, $\pi_*\OO_X = \OO_Y$.
\end{definition}

The fibrations we consider in this section are algebraic parts of our foliations. In particular, general fibers of such  fibrations are either curves or surfaces of general type. These fibrations belong to a wider class considered by Kawamata in his work \cite{Kaw85} on the Iitaka conjecture. Namely, he considered fibrations with the geometric generic fiber $\overline{X_{\eta}}$ having a good minimal model. For these fibrations it is possible to define the \emph{birational variation} and to compare this invariant to the Kodaira dimension of the relative canonical bundle.

\begin{proposition}{\cite[Theorem 7.2]{Kaw85}} Let $\pi \colon X \to Y$ be a fibration such that the geometric generic fiber $\overline{X_{\eta}}$ has a good minimal model. We call a minimal closed field of definition of $\pi$ a minimal element in the set of all algebraically closed subfields $K \subset \overline{k(Y)}$ satisfying the condition $$\mathrm{Frac}(L \otimes_K \overline{k(Y)}) \cong \mathrm{Frac}(k(X) \otimes_{k(Y)} \overline{k(Y)}) \quad \mbox{over $\overline{k(Y)}$}$$ for some finitely generated extension $L \supset K$. Then a minimal closed field of definition exists and is unique.
\end{proposition}

\begin{definition} Let $\pi \colon X \to Y$ be a fibration as above. We define the \emph{birational variation} $\Var(\pi)$ as transcendence degree over $\mathbb{C}$ of the minimal closed field of definition of $\pi$. 
\end{definition}

\begin{definition}\label{isotrivial} A fibration $\pi \colon X \to Y$ such that $\Var(\pi) = 0$ is called a {\em birationally isotrivial} fibration.
\end{definition}

\begin{remark}\label{Var} Suppose that a general fiber of $\pi \colon X \to Y$ is a curve of genus at least 2 and that $\pi$ is a stable fibration. Then by the main results of \cite{KM76, Knu83} there exists a coarse moduli space $\mathcal{M}$ for the fibers. Then $\Var(\pi)$ is equal to the variation in the sense of moduli theory, that is, the dimension of the image of $Y$ under the moduli map corresponding to $\pi$.
\end{remark}

\begin{theorem}{\cite[Theorem 1.1]{Kaw85}} \label{Kaw1} Let $\pi \colon X \to Y$ be a fibration between smooth projective varieties. Suppose that the geometric generic fiber $X_{\eta}$ of $\pi$ has a good minimal model. Then the following inequalities hold:
\begin{enumerate}
\item $\kappa(Y, \det(\pi_*\OO_X(mK_{X/Y}))) \geqslant \Var(\pi)$ for some $m \in \mathbb{N}$;
\item If $L$ is a line bundle on $Y$ such that $\kappa(Y, L) \geqslant 0$, then $$\kappa(X, \OO_X(K_{X/Y}) \otimes \pi^*L) \geqslant \kappa(X_{\eta}) + \max\{\kappa(Y, L), \Var(\pi)\}.$$
\end{enumerate}
\end{theorem}

\begin{corollary}{\cite[Corollary 1.2]{Kaw85}} \label{cor:Kaw2} In the assumptions of Theorem \ref{Kaw1}, let $F$ be a general fiber of $\pi$. Then we have the inequality $$\kappa(X, \OO_X(K_{X/Y})) \geqslant \kappa(K_F) + \Var(\pi).$$
\end{corollary}

We recall here a very important construction for fibrations which allows to eliminate multiple fibers in codimension ~1 by a generically finite base change. The exposition below follows the one given in the proof of \cite[Theorem 7.1]{CKT16}. 

\begin{construction} \label{CKT} Let $f \colon X \dasharrow Y$ be a rational map with connected fibers between normal projective varieties. Assume that there exists a big open subset $X^{\circ} \subset X$ such that $f|_{X^{\circ}}$ is an equidimensional morphism. Then there exists a diagram
\begin{equation}\label{semred}
\xymatrix{
\overline{X} \ar[rr]^{a} \ar[d]^{\overline{f}} && \widetilde{X} \ar[r]^b \ar[d] & X \ar@{-->}[d]^f\\
\overline{Y} \ar[r]_{\alpha} & \widetilde{Y} \ar[r]_{\beta} & Y \ar@{=}[r] & Y
}
\end{equation}
where the maps are as follows:
\begin{itemize}
\item $b \colon \widetilde{X} \to X$ is a strong log resolution of indeterminacies of $f$ and of singularities of $X$;
\item $\beta \colon \widetilde{Y} \to Y$ is an adapted Galois cover of the pair $(Y, B)$ (see \cite[Proposition 2.38]{CKT16}), where $B$ is the orbifold branch divisor (see \cite[Definition 2.24]{CKT16}) of the map $f$;
\item $\alpha \colon \overline{Y} \to \widetilde{Y}$ is a log resolution of the pair $(\widetilde{Y}, \beta^*B)$;
\item $a \colon \overline{X} \to \widetilde{X}$ is a log resolution of the main component of the fiber product $\overline{Y} \times_{Y}\widetilde{X}$.
\end{itemize}

As a result, we obtain a fibration of smooth varieties $\overline{f} \colon \overline{X} \to \overline{Y}$ and generically finite morphisms $b\circ a \colon \overline{X} \to X$ and $\beta \circ \alpha \colon \overline{Y} \to Y$. Moreover, (see \cite[Observation 7.5]{CKT16}) there exist big open subsets $X^{\circ} \subset X$ and $Y^{\circ} \subset Y$ such that the preimage $\overline{X}^{\circ} = (b \circ a)^{-1}(X^{\circ})$ is also big and such that the following formula holds (see \cite[Consequence 7.8]{CKT16}): $$K_{\overline{X}^{\circ}/\overline{Y}^{\circ}} \sim_{\QQ} (b \circ a)^*(K_{X^{\circ}/Y^{\circ}} - R(f)).$$ 
\end{construction}

\subsection{The classification: positive algebraic rank} 

\begin{proposition} \label{AlgInt}
Let $\FF$ be a codimension 1 foliation with canonical singularities on a smooth projective threefold $X$ of general type. Suppose that $\FF$ is algebraically integrable and that the canonical class $K_{\FF}$ is not big. Then there exists a generically finite morphism $f \colon X' \to X$ such that the threefold $X'$ is birational to a product $S \times C$ of a surface $S$ and a curve $C$ and the foliation $f^{-1}\FF$ is birationally equivalent to the relative tangent bundle of the projection $S \times C \to C$.
\end{proposition}

\begin{proof}By Proposition \ref{ratint} there is a rational map inducing the foliation $\FF$. Consider a fibration $\pi \colon \widetilde X \to C$ which is obtained by resolving indeterminacies of this rational map and applying Proposition \ref{flatten}. We have the induced foliation $\widetilde \FF$ on $\widetilde{X}$. By Proposition \ref{canclassalg} the canonical class of $\widetilde{\FF}$ is given by the formula $$K_{\widetilde \FF} = K_{\widetilde X/C} - R(\pi).$$ 
We apply Construction \ref{CKT} to $\pi \colon \widetilde X \to C$ and obtain a diagram
\begin{equation}
\xymatrix{
\overline{X} \ar[r] \ar[d] & \overline{C} \ar[d] \\
\widetilde{X} \ar[r] & C
}
\end{equation}
 and a foliation $\overline{\FF}$ on $\overline{X}$.
Then we can apply Corollary \ref{cor:Kaw2} to the fibration $\overline{\pi}$ and obtain $$2 \geqslant \nu(K_{\widetilde{\FF}}) = \nu(K_{\widetilde{X}/C} - R(\pi)) = \nu(K_{\overline{X}^{\circ}/\overline{C}^{\circ}}) = \nu(K_{\overline{X}/\overline{C}}) \geqslant \kappa(K_F) + \Var(\overline{\pi}).$$ By adjunction, a general fiber $F$ of $\pi$ is a surface of general type, that is, $\kappa(K_F) = 2$. Therefore we have $\Var(\overline{\pi}) = \Var(\pi) = 0$ and by definition of birational variation, there exists a generically finite morphism $f \colon X' \to X$ such that $X'$ is birational to a product $S' \times C'$ of a surface $S'$ of general type and a curve $C'$ of genus at least 2. Moreover, the foliation $f^{-1}\FF$ is birationally equivalent to the foliation induced by the projection $S' \times C' \to C'$. The proposition is proved. 
\end{proof}

Next, we treat the case of foliations with algebraic rank $1$. The idea is the same as in the proof of Proposition \ref{AlgInt}. We express the canonical class of $\FF$ in terms of the canonical classes of its algebraic and transcendental parts using Proposition \ref{canclasspull}. Then we use our assumption $\nu(K_{\FF}) < 3$ together with Construction \ref{CKT} and Theorem \ref{Kaw1} in order to obtain restrictions on the birational variation of the algebraic reduction $\pi$ of $\FF$. The rest of the proof is case-by-case analysis according to possible values of $\Var(\pi)$ and $\nu(K_{\GG})$.

\begin{theorem} \label{MainThm} Let $\FF$ be a codimension 1 foliation with canonical singularities on a smooth projective threefold $X$ of general type. Suppose that the algebraic rank of $\FF$ is equal to 1 and that the canonical class of $\FF$ is not big. Then there exists a birational morphism $f \colon X' \to X$ such that the threefold $X'$ admits a birationally isotrivial fibration $\pi \colon X' \to S$ to a surface $S$, and the foliation $f^{-1}\FF$ is the pullback via $\pi$ of a foliation $\GG$ on $S$ such that $\nu(K_{\GG}) = 1$. Moreover, if $\FF$ is non-abundant then $\GG$ is birationally equivalent to a Hilbert modular foliation. 
\end{theorem}

\begin{proof} By Proposition \ref{algred} we may consider the algebraic reduction of $\FF$. After resolving the indeterminacies and singularities and applying Proposition \ref{flatten}, we obtain an equidimensional fibration $\pi \colon X \to S$ from a smooth projective threefold (which we also denote by $X$) of general type to a smooth surface $S$. Moreover, the foliation $\FF$ is the pullback of a transcendental foliation $\GG$ on $S$. By Proposition \ref{canclasspull} the canonical class of $\FF$ can be expressed by the formula $$K_{\FF} = \pi^*K_{\GG} + K_{X/S} - R(\pi)^{\inv}.$$ Here for convenience of notation we denote $$R(\pi)^{\inv} = \sum_{B}(\pi^*B - (\pi^*B)_{\mathrm{red}}),$$ where the sum is over all $\GG$-invariant prime divisors $B$ on $S$ contained in the critical locus of $\pi$. Since $\GG$ is transcendental, the canonical class $K_{\GG}$ is pseudoeffective (see e.g. \cite[Theorem 7.1]{Bru15}). By \cite[Theorem 1.3]{CP19}, the canonical class $K_{X/S} - R(\pi)$ of the algebraic part of $\FF$ is pseudoeffective as well. Moreover, by our assumption on $\FF$ we have the inequalities $$\nu(K_{X/S} - R(\pi)) \leqslant \nu(K_{X/S} - R(\pi)^{\inv}) \leqslant \nu(K_{\FF}) \leqslant 2.$$

On the other hand, we apply Construction \ref{CKT} to obtain a diagram as in \eqref{semred}:
\begin{equation}
\xymatrix{
\overline{X} \ar[r]^{\overline{\pi}} \ar[d]_{f} & \overline{S} \ar[d]^{g}\\
X \ar[r]^{\pi} & S
}
\end{equation}
and induced foliations $\overline{\FF}$ and $\overline{\GG}$ on $\overline{X}$ and $\overline{S}$, respectively.
We can use this diagram to estimate the variation of the fibration $\pi$ from above. Indeed, by Corollary \ref{cor:Kaw2} we obtain $$\nu(K_{X/S} - R(\pi)) = \nu(K_{\overline{X}^{\circ}/\overline{S}^{\circ}}) \geqslant \kappa(K_{\overline{X}^{\circ}/\overline{S}^{\circ}}) = \kappa(K_{\overline{X}/\overline{S}}) \geqslant \kappa(K_{F}) + \Var(\overline{\pi}) = 1 + \Var(\pi),$$ since a general fiber $F$ of $\pi$ is a curve of general type. Thus we are left with two possibilities: either $\Var(\pi) = 0$ or $\Var(\pi) = 1$.

\medskip

Moreover, we can derive a bound on the numerical dimension of $K_{\GG}$. Namely, we have $$f^*K_{\FF} = f^*(\pi^*K_{\GG} + K_{X/S} - R(\pi)^{\inv}) \geqslant f^*(\pi^*K_{\GG} + K_{X/S} - R(\pi)) = K_{\overline{X}/\overline{S}} + (g \circ \overline{\pi})^*K_{\GG}.$$ We can then apply part (2) of Theorem \ref{Kaw1} to the fibration $\overline{\pi}$, taking $L = g^*\OO_{S}(K_{\GG})$, and obtain $$2 \geqslant \kappa(f^*K_{\FF}) \geqslant \kappa((g \circ \overline{\pi})^*K_{\GG} + K_{\overline{X}/\overline{S}}) \geqslant 1 + \max\{\Var(\pi), \kappa(g^*K_{\GG})\},$$ provided that $\kappa(K_{\GG}) \geqslant 0$. Therefore, the foliation $\GG$ is not of general type. To complete the proof, we need to sort out the following cases: \begin{itemize} \item $\nu(K_{\GG}) = 1$ and $\Var(\pi) = 0$; \item $\nu(K_{\GG}) = 1$ and $\Var(\pi) = 1$; \item $\nu(K_{\GG}) = 0$ and $\Var(\pi) \leqslant 1$. \end{itemize}
For an example of a foliation with $\nu(K_{\GG}) = 1$ and $\Var(\pi) = 0$ take $X = S \times C$ where $C$ is a smooth curve of genus at least 2, $S$ is a Hilbert modular surface and $\GG$ is a Hilbert modular foliation. Define the foliation $\FF$ to be the pullback of $\GG$ by the projection $S\times C \to S$. Then $\nu(K_{\FF}) = \nu(K_{\GG} \boxtimes K_C) = 2$. 

In the remaining part of the proof we show that the other two cases do not occur.

\medskip

If $\nu(K_{\GG}) = 0$ then by the classification theorem of McQuillan \cite[Theorem 2 IV.3.6]{McQ08} (see also \cite[Theorem 8.2]{Bru15}), there exists a finite morphism $g \colon \overline{S} \to S$ and a birational contraction $\varphi \colon \overline{S} \to S'$ such that the foliation $\GG'$ on $S'$ is given by a holomorphic vector field with isolated zeroes. Moreover, by {\em loc. cit.} (see also \cite[Propositions 6.5 and 6.6]{Bru15}) we have one of the following possibilities:
\begin{enumerate}
\item $(S', \GG')$ is isomorphic to a quotient $(E \times C)/G$ where $E$ is an elliptic curve, $C$ is a curve of genus at least 2, with a foliation $\GG$ induced by the $G$-invariant fibration $C \times E \to C$;
\item $(S', \GG')$ is isomorphic to abelian surface with a linear foliation;
\item $(S', \GG')$ is isomorphic to a smooth $\PP^1$-bundle over an elliptic curve $E$ with a Riccati foliation;
\item $(S', \GG')$ is isomorphic to a $\PP^1$-bundle over $\PP^1$ with a Riccati foliation.
\end{enumerate}

As $\GG$ is transcendental, the case (1) does not occur. In each of the three remaining cases there are only finitely many $\GG'$-invariant curves $B_1, \ldots, B_k$ on $S'$. 

Now consider the fibration $\pi \colon X \to S$ and a foliation $\GG$ such that $\nu(K_{\GG}) = 0$. We have the finite morphism $g \colon \overline{S} \to S$ such that $g^*K_{\GG} = K_{\overline{\GG}}$ by \cite[Theorem 8.2]{Bru15}. Thus we can consider a generically finite base change $\overline{\pi} \colon \overline{X} \to \overline{S}$ such that the $K_{\overline{\GG}}$ has a holomorphic section (possibly with zeroes) and $K_{\overline{\FF}}$ is not big by Proposition \ref{hurwitz}. Let also $\varphi \colon \overline{S} \to S'$ be a birational contraction such that $K_{\GG'} = \OO_{S'}$. Then $(S', \GG')$ belongs to one of the types (2) -- (4) of the above classification. In particular (see \cite[Theorem 2 IV.3.6]{McQ08}), the log pair $(S, B')$ where $B' = \sum B_i$ is an equivariant compactification of a connected algebraic group $G$ and the foliation $\GG'$ is induced by a one-parameter subgroup of $G$. It follows (see e.g. \cite[3.1]{BZ17}) that the line bundle $$\det(T_{S'}(-\log B')) = -(K_{S'} + B')$$ is effective. Hence, the line bundle $-(K_{\overline{S}} + B) = -\varphi^*(K_{S'} + B')$ is effective as well.

We can express the canonical class of $\overline{X}$ as $$K_{\overline{X}} = (K_{\overline{X}/\overline{S}} - \overline{\pi}^*B) + \overline{\pi}^*(K_{\overline{S}} + B).$$ 
Here $K_{\overline{X}}$ is big, whereas $$\nu(K_{\overline{X}/\overline{S}} - \overline{\pi}^*B) \leqslant \nu(K_{\overline{\FF}}) = \nu(K_{\FF}) = \nu(K_{X/S} - R(\pi)^{\inv}) \leqslant 2.$$ Since $-(K_{\overline{S}} + B)$ is effective, the class at the right hand side of the formula is not big. This is a contradiction, so the case $\nu(K_{\GG}) = 0$ is impossible.

\medskip

Next, we need to exclude the case $\nu(K_{\GG}) = 1$ and $\Var(\pi) = 1$. In this case we have $$\nu(K_{X/S} - R(\pi)^{\inv}) = 2.$$

Then using superadditivity of the restricted positive product from Theorem \ref{posprod} we obtain 
\begin{multline}\label{zerovol}
0 = \mathrm{vol}(K_{\FF}) = \langle K_{\FF}\rangle^3 = \langle K_{X/S} - R(\pi)^{\inv} + \pi^*K_{\GG}\rangle^3 \geqslant \\ \geqslant \langle\pi^*K_{\GG}\rangle^3 + 3\langle\pi^*K_{\GG}\cdot (K_{X/S} - R(\pi)^{\inv}) \cdot (K_{X/S} - R(\pi)^{\inv})\rangle +3\langle\pi^*K_{\GG}^2\cdot (K_{X/S} - R(\pi)^{\inv})\rangle +\langle K_{X/S} - R(\pi)^{\inv}\rangle^3 = \\  = 0 + 0 + 0 + 3\langle\pi^*K_{\GG}\cdot (K_{X/S} - R(\pi)^{\inv}) \cdot (K_{X/S} - R(\pi)^{\inv})\rangle.
\end{multline} 

By the results of \cite{dJ97} (see also \cite{AK00}), we can construct a generically finite base change 
\begin{equation}
\xymatrix{
\overline{X} \ar[r]^{\overline{\pi}} \ar[d]_{f} & \overline{S} \ar[d]^{h}\\
X \ar[r]^{\pi} & S
}
\end{equation}

where the morphism $\overline{\pi} \colon \overline{X} \to \overline{S}$ is a family of semistable curves in codimension 1. Then by \cite[Theorem 9.31]{Vie95} we have a moduli map $\mu \colon \overline{S} \to \overline{\mathcal{M}_g}$ such that $\dim(\mu(\overline{S})) = \Var(\pi) = 1$ (see Remark \ref{Var}). Moreover, for any curve $C \subset \overline{S}$ the pullback to $C$ of the polarization on $\overline{\mathcal{M}_g}$ is given by the restriction of $\det(\overline{\pi}_*\OO_X(mK_{\overline{X}/\overline{S}}))$ for some $m > 0$. We have the following equality of movable intersection numbers \begin{multline} \langle\pi^*K_{\GG}\cdot (K_{X/S} - R(\pi)^{\inv}) \cdot (K_{X/S} - R(\pi)^{\inv})\rangle = \langle(\pi \circ f)^*K_{\GG}\cdot f^*(K_{X/S} - R(\pi)^{\inv}) \cdot f^*(K_{X/S} - R(\pi)^{\inv})\rangle = \\ = \langle(h \circ \overline{\pi})^*K_{\GG} \cdot K_{\overline{X}/\overline{S}} \cdot K_{\overline{X}/\overline{S}}\rangle,\end{multline} see e. g. the argument in \cite[Proposition 3.3]{LX17}. Now using the assumptions $\Var(\pi) = 1$ and $\nu(K_{\GG}) = 1$ we will estimate the movable intersection number $$\langle(h \circ \overline{\pi})^*K_{\GG} \cdot K_{\overline{X}/\overline{S}} \cdot K_{\overline{X}/\overline{S}}\rangle$$ from below and arrive to a contradiction with equality \eqref{zerovol}. By \cite[Theorem I]{Sch12} there exists a singular metric on $K_{\overline{X}/\overline{S}}$ such that \begin{itemize} \item its curvature current $\omega_{\overline{X}/\overline{S}}$ has analytic singularities contained in the singular locus $\mathrm{Sing}(\overline{\pi})$; \item the absolutely continuous part $(\omega_{\overline{X}/\overline{S}})_{ac}$ in the Lebesgue decomposition of $\omega_{\overline{X}/\overline{S}}$ (see e. g. \cite[Section 2.3]{Bou02}) is smooth on $X^{\circ} = \overline{X}\setminus \mathrm{Sing}(\overline{\pi})$; \item the integral of $(\omega_{\overline{X}/\overline{S}})_{ac}$ along the fibers of $\overline{\pi}|_{X^{\circ}}$ is equal to the (generalized) Weil--Petersson form $\omega_{WP}$ on the base of the family. \end{itemize} We also consider the nef class $P_{\GG}$ which is the positive part in the Zariski decomposition of $K_{\GG}$ and a closed positive current $\omega_{\GG}$ in the class $P_{\GG}$.

Let $U$ be a Zariski-open subset of $\overline{S}$ and denote $X^{\circ} = \overline{\pi}^{-1}(U)$. We may choose $U$ in such a way that \begin{itemize} \item $\overline{\pi}|_{X^{\circ}} \colon X^{\circ} \to U$ is a smooth family of curves of genus $g \geqslant 2$; \item the restriction of $h$ to $U$ is finite and unramified. \end{itemize} Then from the definition of the movable product and from the choice of $U$ we obtain

\begin{equation} \label{zeroeq}
0 = \langle(h \circ \overline{\pi})^*K_{\GG} \cdot K_{\overline{X}/\overline{S}} \cdot K_{\overline{X}/\overline{S}}\rangle \geqslant  \int_{X^{\circ}} ((h \circ \overline{\pi})^*\omega_{\GG})_{ac} \wedge (\omega_{\overline{X}/\overline{S}})_{ac} \wedge (\omega_{\overline{X}/\overline{S}})_{ac} \geqslant \int_U \omega_{WP} \wedge (h^*\omega_{\GG})_{ac}.
\end{equation}

Let us denote by $C$ a general fiber of the moduli map. If $h^*P_{\GG}\cdot C > 0$ then the latter integral in \eqref{zeroeq} is positive by integration along the fibers of $\mu|_{U}$. So we need to exclude the case when $h^*P_{\GG}\cdot C = 0$. In that case there exists a curve $F = h_*C$ on $S$ such that deformations of $F$ cover $S$ and $P_{\GG}\cdot F = 0$. By our assumption, we have $\nu(K_{\GG}) = 1$ and $\GG$ is transcendental, so by \cite[Theorem 9.1]{Bru15} the foliation $\GG$ is birational to a Riccati or a turbulent foliation and $F$ is a general fiber of the associated Iitaka fibration.

On the other hand, we have $$K_X = (K_{X/S} - \pi^*B) + \pi^*(K_S + B),$$ where $B = \sum_iB_i$ is a sum of $\GG$-invariant curves in the branch locus of $\pi$. Let $F$ is a general fiber of the Iitaka fibration of $K_{\GG}$. Then for every $m > 0$ we have the equality $$c_1(\pi_*\OO_X(mK_X)) \cdot F = c_1(\pi_*\OO_X(m(K_{X/S} -\pi^*B))) \cdot F + c_1(\OO_S(K_S + B)) \cdot F.$$ Since $K_X$ is big and $F$ is general, we have $c_1(\pi_*\OO_X(mK_X)) \cdot F > 0$. We also have $K_{\GG} \cdot F = 0$ and $N^*_{\GG} \cdot F = \deg(K_F)$ by definition. Moreover, the intersection $B\cdot F$ is equal to zero for $\GG$ a turbulent foliation (all $\GG$-invariant curves are contained in fibers) and does not exceed 2 for $\GG$ a Riccati foliation ($\GG$ has at most 2 invariant sections, up to linear equivalence). Therefore we have $$(K_S + B)\cdot F = (K_{\GG} + N^*_{\GG} + B) \cdot F \leqslant 0.$$Thus we obtain $$c_1(\pi_*\OO_X(m(K_{X/S} -\pi^*B))) \cdot F = \det(\pi_*\OO_X(m(K_{X/S} -\pi^*B))) \cdot F > 0.$$ Pulling back to $\overline{S}$ and restricting to $U$ we obtain $$\int_{C \cap U}\omega_{WP} = \det(\overline{\pi}_*\OO_X(mK_{\overline{X}/\overline{S}}))\cdot C \geqslant \det(\pi_*\OO_X(m(K_{X/S} -\pi^*B))) \cdot F > 0$$ which contradicts equality \eqref{zeroeq}. Therefore the case $\Var(\pi) = 1$ and $\nu(K_{\GG}) = 1$ is also impossible. 

\medskip

Finally, suppose that $\GG$ is abundant, that is, $\kappa(K_{\GG}) = \nu(K_{\GG}) = 1$. Then by Theorem \ref{Kaw1} we have $$2 \geqslant \kappa(K_{\FF}) = \kappa((g \circ \overline{\pi})^*K_{\GG} + K_{\overline{X}/\overline{S}}) \geqslant \kappa(F) + \kappa(K_{\GG}) = 2$$ so $\FF$ is abundant as well. Therefore $\FF$ can be non-abundant only if $\GG$ is non-abundant. By \cite[Proposition 9.2]{Bru15} the foliation $\GG$ is birational to a Hilbert modular foliation. The theorem is proved.
\end{proof}

\begin{remark} In Theorem \ref{MainThm} we obtained that the fibration $\pi \colon X' \to S$ is birationally isotrivial, hence there exists a generically finite morphism $\overline{X} \to X'$ such that $\overline{X}$ is birational to a product $C \times \overline{S}$. However, unlike in Proposition \ref{AlgInt}, the morphism $\pi$ can be ramified over non-$\GG$-invariant divisors. Thus by Proposition \ref{hurwitz} the induced foliations on $\overline{X}$ and $\overline{S}$ can be of general type.
\end{remark}

\section{The case of purely transcendental foliations} \label{main2}  In this section we consider purely transcendental foliations on smooth projective varieties of general type. Assuming that the foliation is non-singular in codimension $2$, we can obtain the following description for these foliations.

\begin{theorem} \label{terminal} Let $X$ be a smooth projective manifold of general type, $\dim(X) = n \geqslant 2$. Let $\FF$ be a codimension 1 foliation on $X$. Suppose that \begin{enumerate} 
\item $K_{\FF}$ is not big;
\item $\FF$ is purely transcendental;
\item $\mathrm{codim}_X\Sing(\FF) \geqslant 3$.
\end{enumerate}

Then the foliation $\FF$ is induced by a Hilbert modular foliation via a morphism $X \to M_H$, generically finite onto its image. 
\end{theorem}

To obtain the conclusion of Theorem \ref{terminal} we show that a foliation $\FF$ satisfying the assumptions of the theorem belongs to the class of foliations with {\em pseudoeffective} and {\em non-abundant} conormal bundle, studied by Touzet in \cite{Tou13, Tou16}. He proved the following structure theorem for foliations from this class (see \cite[Theorem 1]{Tou16}).

\begin{theorem}\label{touzet} Let $\FF$ be a codimension 1 foliation on a smooth projective variety $X$. Suppose that the conormal bundle $N^*_{\FF}$ is pseudoeffective and that $\kappa(N^*_{\FF}) = - \infty$. Then there exists an analytic morphism $\Psi \colon X \to M_H$ to a Hilbert modular variety $M_H \simeq \mathbb{D}^N/\Gamma$ (here $N \geqslant 2$) such that $\FF = \Psi^{-1}\GG$ for $\GG$ being one of the Hilbert modular foliations on $M_H$.
\end{theorem}

To obtain the morphism $\Psi$ to a Hilbert modular variety Touzet uses the following theorem of Corlette and Simpson (see \cite[Theorem 2]{CS08}).

\begin{theorem}\label{corsim} Let $U = X\setminus D$ be a smooth quasiprojective variety and let $\rho \colon \pi_1(U, x) \to \mathrm{SL}_2(\mathbb{C})$ be a representation with Zariski-dense image. Suppose that the monodromy transformations around the components of $D$ are quasi-unipotent. Then either $\rho$ comes from a map $f \colon U \to C$ to a Deligne--Mumford curve or $\rho$ is a pullback of one of the tautological representations by a map $f \colon U \to M_H$ to a Hilbert modular variety.
\end{theorem}

Before we proceed to the proof of Theorem \ref{terminal}, we state and prove a technical lemma.

\begin{lemma}\label{restcurves} Let $\FF$ be a codimension 1 purely transcendental foliation on a smooth projective variety $X$. Let $H \subset X$ be a very general hyperplane section. If $C \subset H$ is a compact connected curve tangent to $\FF$, then the leaf of $\FF$ containing the curve $C$ is algebraic.
\end{lemma}

\begin{proof} Since $\FF$ is purely transcendental, for any $\FF$-invariant compact curve $C \subset X$ its $\FF$-invariant deformations inside $X$ cover a proper subset $W_C$ of $X$. Moreover, $\FF$-invariant curves $C$ not contained in any $\FF$-invariant hypersurface are parameterized by at most countably many components in the Hilbert scheme. Thus we can choose a very general hyperplane section $H$ which does not contain any of these curves. 
\end{proof}

Now we can prove Theorem \ref{terminal}.

\begin{proof}[Proof of Theorem \ref{terminal}] If $\dim(X) = 2$ then by assumption $\FF$ is regular, therefore canonical (\cite[Lemma 3.10]{AD13}). The statement then follows from Theorem \ref{BruMcQ} with the morphism $X \to M_H$ being the minimal model of $\FF$. 

Suppose from now on that $\dim(X) = n \geqslant 3$. 

\emph{Step 1: Find a suitable complete intersection surface.} By our assumptions, $K_X$ is big and $K_{\FF}$ is not big. From Proposition \ref{Numprop} and from the formula $$K_{\FF} = K_X + N_{\FF}$$ we obtain that the normal bundle $N_{\FF}$ is not pseudoeffective. Since $X$ is smooth, by \cite[Theorem ~1.5]{BDPP13} this is equivalent to the following condition: there exists a birational model $\varphi \colon X' \to X$ and a complete intersection class $\alpha = H_1 \cap \cdots \cap H_{n-1}$ on $X'$ such that \begin{equation} N_{\FF} \cdot \varphi_*(H_1 \cap \cdots \cap H_{n-1}) =  \varphi^*N_{\FF} \cdot H_1 \cdots H_{n-1} < 0. \label{norm}\end{equation} By Bertini's theorem we can take the above ample divisors $H_2, \ldots, H_{n-1}$ and large multiples $m_i$ for $i \in \{2, \ldots, n-1\}$ such that a general element $$D \in |m_2H_2 \cap \cdots \cap m_{n-1}H_{n-1}|$$ is a smooth surface. Moreover, by Lemma \ref{restcurves} every compact curve $C \subset D$, invariant under $\FF'_D$, is an intersection $C = D \cap E$ for an $\FF'$-invariant hypersurface $E \subset X'$. Note that Lemma \ref{restcurves} implies that the number of $\FF'_D$-invariant curves is finite. Indeed, if the number of $\FF'_D$-invariant curves is infinite, then the number of $\FF'$-invariant surfaces $E \subset X'$ is infinite as well. Therefore by Jouanolou's theorem \cite{Jou78} the foliation $\FF'$ has to be algebraically integrable, which is not the case.

\emph{Step 2: Prove pseudoeffectivity of $\varphi^*N^*_{\FF}$ on $D$.} Recall from Remark \ref{pullback} that we have the relation $$N^*_{\FF'} = \varphi^*N^*_{\FF} + E$$ where $E$ is an effective $\varphi$-exceptional divisor. Therefore, if $(\varphi^*N^*_{\FF})|_D$ is pseudoeffective then the same is true for $N^*_{\FF'_D}$. Let us consider the following $\QQ$-divisors on $X'$: $$L_{\varepsilon} := \varphi^*N^*_{\FF} + \varepsilon H_1, \quad \varepsilon \in \QQ_{>0}.$$ By the Riemann--Roch theorem we obtain $$h^0(D, mL_{\varepsilon}|_D) + h^2(D, mL_{\varepsilon}|_D) \geqslant C_{\varepsilon}\cdot(L_{\varepsilon}|_D)^2m^2$$ for a positive constant $C_{\varepsilon}$ and $m$ large and sufficiently divisible. The intersection number is equal to \begin{multline} (L_{\varepsilon}|_D)^2 = (\varphi^* N^*_{\FF}|_D + \varepsilon H_1|_D)^2 = ((\varphi^* N^*_{\FF})^2 + 2\varepsilon \varphi^*  N^*_{\FF} \cdot H_1 + \varepsilon^2 H_1^2)\cdot D = \\ = (\varphi^* N^*_{\FF})^2\cdot m_2 H_2 \cdots m_{n-1} H_{n-1} + 2\varepsilon \varphi^* N^*_{\FF} \cdot H_1 \cdot m_2H_2 \cdots m_{n-1}H_{n-1} + \varepsilon^2 H_1 \cdot m_2 H_2 \cdots m_{n-1} H_{n-1}.\end{multline} The second summand in the above equality is positive by the condition \eqref{norm}. The third one is also positive, since $H_i$ are ample divisors. As for the first summand, by the projection formula we have $$(\varphi^*N^*_{\FF})^2 \cdot H_2 \cdots H_{n-1} = N^*_{\FF} \cdot \varphi_*(\varphi^*N^*_{\FF} \cdot H_2 \cdots H_{n-1}) = N^2_{\FF} \cdot \varphi_*(H_2 \cap \cdots \cap H_{n-1}).$$ By our assumption $\mathrm{codim}_X\Sing(\FF) \geqslant 3$ and by the Baum--Bott formula we have 
\begin{equation}N^2_{\FF} \cdot \varphi_*(H_2 \cap \cdots \cap H_{n-1}) = 0 \quad \mbox{since} \quad N^2_{\FF} \equiv 0.\label{BBott}\end{equation} Thus $(L_{\varepsilon}|_D)^2 > 0$ for all $\varepsilon \in \QQ_{>0}$. Moreover, by Serre duality and by the condition \eqref{norm} we have $$h^2(D, mL_{\varepsilon}|_D) = h^0(D, K_D + m\varphi^*N_{\FF} - m\varepsilon H_1) = 0$$ for $m$ large enough. Therefore for every $\varepsilon \in \QQ_{>0}$ we obtain $$h^0(D, mL_{\varepsilon}|_D) > C \cdot (L_{\varepsilon}|_D)^2 m^2 > 0$$ for $m$ large and sufficiently divisible (depending on $\varepsilon$). So the class of $L_0|_D = \varphi^*N^*_{\FF}|_D$ is a limit of classes of $\QQ$-effective divisors, hence it is pseudoeffective. Then the conormal line bundle  $$N^*_{\FF'_D} = N^*_{\FF'}|_D = (\varphi^* N^*_{\FF} + E)|_D$$ is pseudoeffective as well.

\emph{Step 3: Case-by-case analysis.}
By the classification result of Touzet \cite[Proposition 2.14]{Tou13}, we have three possibilities for the numerical and Kodaira dimensions of $N_{\FF'}^*|_D$: \begin{enumerate}
\item $\nu(N_{\FF'}^*|_D) = 0$;
\item $\nu(N_{\FF'}^*|_D) = 1, \kappa(N_{\FF'}^*|_D) = 1$;
\item $\nu(N_{\FF'}^*|_D) = 1, \kappa(N_{\FF'}^*|_D) = -\infty$.
\end{enumerate}

We consider these three cases separately. 

\textit{Case $\nu(N_{\FF'}^*|_D) = 0$}. Since $\varphi^*N^*_{\FF}|_D$ is pseudoeffective and $N^*_{\FF'}|_D = (\varphi^*N^*_{\FF} + E)|_D$, it follows that $\nu(\varphi^*N^*_{\FF}|_D) = 0$. We consider the Zariski decomposition $(\varphi^*N_{\FF}^*)|_D = L + \sum_ia_iC_i$, where $L$ is numerically trivial and $C_i$ are exceptional curves on $D$. We obtain \begin{equation}\begin{split}(\varphi^*N^*_{\FF}|_D)^2 &= (\varphi^*N^*_{\FF})^2\cdot m_2H_2 \cdots m_{n-1}H_{n-1} = 0 \qquad \mbox{(by the equality  \eqref{BBott})} \\ &= L^2 + (\sum a_iC_i)^2 = (\sum a_iC_i)^2.\end{split}\end{equation} Since the Gram matrix of $\{C_i\}$ is negative definite, we obtain that $C_i = 0$, so that $\varphi^*N^*_{\FF}|_D = L \equiv 0$. By the Lefschetz hyperplane section theorem the map $i^* \colon H^2(\widetilde X, \mathbb{C}) \to H^2(D, \mathbb{C})$ is injective, therefore we have $\varphi^*N^*_{\FF} \equiv 0$ on $X'$. However, this implies $\nu(\varphi^*K_{X}) = \nu(\varphi^*K_{\FF}) = n$, which contradicts our assumptions. 

\textit{Case $\nu(N_{\FF'}^*|_D) = 1, \kappa(N_{\FF'}^*|_D) = 1$}. In this case we apply a result of Bogomolov \cite[Lemma 12.4]{Bog79} 
and obtain that $\FF'_D$ is algebraically integrable. However, the algebraic reduction of $\FF'_D$ is induced by that of $\FF'$ (see \cite[Lemma 4]{PS20}), so the algebraic rank of $\FF'$ has to be positive, which is not the case by assumption.

\textit{Case $\nu(N_{\FF'}^*|_D) = 1, \kappa(N_{\FF'}^*|_D) = -\infty$}. By Theorem \ref{touzet} there exists a map $\Psi_D \colon D \to M_H$, where $M_H = \mathbb{D}^N/\Gamma$ is a Hilbert modular variety, and $\FF'_D$ is the pullback via $\Psi_D$ of one of the Hilbert modular foliations on $M_H$. The map $\Psi_D$ is constructed from a representation $$\rho_D \colon \pi_1(D\setminus \mathrm{Supp}(N_D)) \to SL_2(\mathbb{C})$$ with quasi-unipotent monodromy around the components of $N_D$ using Theorem \ref{corsim}. Here $N_D$ is the negative part in the divisorial Zariski decomposition of $N^*_{\FF'_{D}}$. By Lemma \ref{restcurves} we have $N_D = N_{X'}|_D$ where $N_{X'}$ is a divisor on $X'$ with $\mathrm{Supp}(N_{X'})$ being $\FF'$-invariant. Applying the Lefschetz hyperplane section theorem for quasi-projective varieties \cite[Theorem 1.1.1]{HL85}, we obtain that $$\pi_1(D\setminus(\mathrm{Supp}(N_D))) \simeq \pi_1(X'\setminus\mathrm{Supp}(N_{X'})).$$ Therefore we obtain a representation $$\rho_{X'} \colon \pi_1(X'\setminus\mathrm{Supp}(N_{X'}))  \to \mathrm{SL}_2(\mathbb{C})$$ with quasi-unipotent monodromy around the components of $N_{X'}$. Moreover, $\rho_{X'}$ does not come from a map to a Deligne--Mumford curve since $\rho_D$ does not by \cite[Theorem 6.4]{Tou16}.

Therefore we can apply Theorem \ref{corsim} and the argument from \cite[p. 182]{Tou16} to obtain a map $\Psi_{X'} \colon X' \to M_H$ extending the map $\Psi_D \colon D \to M_H$ and such that $\FF'$ is induced via $\Psi_{X'}$ by one of the Hilbert modular foliations on $M_H$. In particular, the conormal bundle $N^*_{\FF'}$ is pseudoeffective. The map $\Psi_{X'}$ has to be generically finite onto its image, since $\FF'$ is purely transcendental. 

Finally, we have the equality $N^*_{\FF} = \varphi_*N^*_{\widetilde\FF}$, thus the conormal bundle $N^*_{\FF}$ is pseudoeffective. Again, by \cite[Proposition 2.14]{Tou13} one of the following three cases occurs: \begin{enumerate}
\item $\nu(N^*_{\FF}) = 0$;
\item $\nu(N^*_{\FF}) = 1, \kappa(N^*_{\FF}) = 1$;
\item $\nu(N^*_{\FF}) = 1, \kappa(N^*_{\FF}) = -\infty$.
\end{enumerate} Suppose that we have $\nu(N^*_{\FF}) = 0$. Since $\FF$ is non-singular in codimension 2, by the Baum--Bott formula (Proposition \ref{BaumBott}) we obtain that $N^*_{\FF} \equiv 0$. If $\nu(N^*_{\FF}) = \kappa(N^*_{\FF}) = 1$ then by a theorem of Bogomolov \cite[Lemma 12.4]{Bog79} the foliation $\FF$ has to be algebraically integrable, which is not the case by assumption. Therefore the only possible case is $\nu(N^*_{\FF}) = 1$ and $\kappa(N^*_{\FF}) = -\infty$. Applying \cite[Theorem 1]{Tou16} to the foliation $\FF$, we obtain the desired conclusion.
\end{proof}

Finally, we list some questions which are natural to ask in view of our results.

\begin{question} Let $\FF$ be a purely transcendental foliation of codimension $1$ on a threefold $X$ of general type. Suppose that $\FF$ has canonical singularities and $\nu(K_{\FF}) < 3$. Is the conormal (or log conormal, for some $\FF$-invariant boundary) bundle always pseudoeffective in this case? The logarithmic version of Touzet's theorem \cite[Theorem 2]{Tou16} gives an affirmative answer to this question for singular Hilbert modular foliations.
\end{question}

\begin{question} Let $\FF$ a foliation be as in the previous question. Can the numerical dimension be equal to $1$? More generally, if $X$ is an $n$-dimensional variety of general type, what is the minimal numerical dimension of a foliation $\FF$ on $X$?
\end{question}

\begin{question} Let $\FF$ be a codimension $1$ foliation with canonical singularities on a projective threefold. Suppose that $\kappa(K_{\FF}) \geqslant 0$; does it follow that $\kappa(K_{\FF}) = \nu(K_{\FF})$?
\end{question}

\flushleft{Aleksei Golota \\
National Research University Higher School of Economics, Russian Federation \\
Laboratory of Algebraic Geometry, NRU HSE, 6 Usacheva str.,Moscow, Russia, 119048 \\
\texttt{golota.g.a.s@mail.ru, agolota@hse.ru}}

\end{document}